\documentclass[11pt]{amsart}
\usepackage[colorlinks=true,pagebackref,hyperindex,citecolor=green,linkcolor=red]{hyperref}
\usepackage{amsmath}
\usepackage{amsfonts}
\usepackage{amssymb}
\usepackage{color}
\usepackage{tcolorbox}
\usepackage{xcolor}
\usepackage[all]{xy}  
\usepackage{enumerate}
\usepackage{verbatim}
\usepackage[top=1in, bottom=1in, left=1in, right=1in]{geometry}
\usepackage{mathrsfs}

\usepackage{stmaryrd}

\theoremstyle{definition}
\newtheorem{theorem}{Theorem}[section]

\newtheorem{corollary}[theorem]{Corollary}
\newtheorem{lemma}[theorem]{Lemma}
\newtheorem{proposition}[theorem]{Proposition}

\theoremstyle{definition}
\newtheorem{definition}[theorem]{Definition}
\newtheorem{setting}[theorem]{Setting}
\newtheorem{example}[theorem]{Example}

\newtheorem{remark}[theorem]{Remark}
\numberwithin{equation}{subsection}

\newtheorem{theoremx}{Theorem}

\newcommand{\m}{\mathfrak{m}}

\newcommand{\NN}{\mathbb{N}}
\newcommand{\ZZ}{\mathbb{Z}}
\newcommand{\QQ}{\mathbb{Q}}
\newcommand{\FF}{\mathbb{F}}

\newcommand{\Spec}{\operatorname{Spec}}

\newcommand{\Hom}{\operatorname{Hom}}

\newcommand{\Max}{\operatorname{Max}}

\renewcommand{\a}{\mathfrak{a}}

\newcommand{\ds}{\displaystyle}

\newcommand{\p}{\mathfrak{p}}
\newcommand{\q}{\mathfrak{q}}

\newcommand{\n}{\mathfrak{n}}

\newcommand{\lr}[1]{{\langle {#1} \rangle}}
\newcommand{\dif}[2]{^{\lr{#1}_{#2}}}
\newcommand{\derp}[1]{^{\lr{#1}_{p}}}
\newcommand{\difM}[1]{^{\lr{#1}_{\mathrm{mix}}}}
\newcommand{\V}{\mathcal{V}}
\newcommand{\Der}{\mathrm{Der}}
\renewcommand{\geq}{\geqslant}
\renewcommand{\leq}{\leqslant}

\begin{document}
\newcommand{\tens}{\otimes}
\newcommand{\hhtest}[1]{\tau ( #1 )}
\renewcommand{\hom}[3]{\operatorname{Hom}_{#1} ( #2, #3 )}
\renewcommand{\t}{\operatorname{type}}

\title{A Zariski-Nagata theorem for smooth $\ZZ$-algebras}

\author[De Stefani]{Alessandro De Stefani}
\address{Department of Mathematics, University of Nebraska, Lincoln, NE 68588-0130, USA}
\email{adestefani2@unl.edu}

\author[Grifo]{Elo\'isa Grifo}
\address{Department of Mathematics, University of Virginia, Charlottesville, VA 22904-4135, USA}
\email{eloisa.grifo@virginia.edu}
\thanks{}

\author[Jeffries]{Jack Jeffries}
\address{University of Michigan, Department of Mathematics, USA}
\email{jackjeff@umich.edu}
\thanks{The third author was partially supported by NSF Grant DMS~\#1606353.}

\maketitle
\begin{abstract}
In a polynomial ring over a perfect field, the symbolic powers of a prime ideal can be described via differential operators: a classical result by Zariski and Nagata says that the $n$-th symbolic power of a given prime ideal consists of the elements that vanish up to order $n$ on the corresponding variety. However, this description fails in mixed characteristic. In this paper, we use $p$-derivations, a notion due to Buium and Joyal, to define a new kind of differential powers in mixed characteristic, and prove that this new object does coincide with the symbolic powers of prime ideals. This seems to be the first application of $p$-derivations to commutative algebra.
\end{abstract}

\section{Introduction}

The subject of symbolic powers is both a classical commutative algebra topic and an active area of current research. While there are many open problems in the setting of algebras containing a field, even the results that are well-understood for algebras over fields are mostly open for $\ZZ$-algebras and local rings of mixed characteristic. Thanks to the perfectoid spaces techniques of Sch\"{o}lze \cite{Scholze} as applied to commutative algebra by Andr\'{e} and Bhatt, a major advance has happened recently \cite{AndreDirectSummand,BhattDirectSummand}. Ma and Schwede have shown that a theorem of Ein-Lazersfeld-Smith~\cite{ELS} and Hochster-Huneke~\cite{HH} on the uniform containment of symbolic and ordinary powers of ideals holds for regular rings of mixed characteristic \cite{MaSchwedeSymbPowers}.

In this paper, we are interested in generalizing another classical result on symbolic powers to the case of mixed characteristic and smooth $\mathbb{Z}$-algebras: the Zariski-Nagata theorem, which establishes that symbolic powers of prime ideals can be described using differential operators.

Zariski's main lemma on holomorphic functions \cite{Zariski}, together with work by Nagata \cite[p. 143]{Nagata}, states that if $P$ is a prime ideal in a polynomial ring over a field, then
$$P^{(n)} = \bigcap_{\substack{\m \supseteq P \\ \m \textrm{ maximal}}} \m^n.$$
This result was later refined by Eisenbud and Hochster \cite{EisenbudHochster}, and can be rephrased using differential powers of ideals, a fact which was well-known in characteristic $0$ and extended to perfect fields; see \cite{SurveySP}. More precisely, if $R$ is a smooth algebra over a perfect field $K$, and $Q$ is a prime ideal in $R$, the Zariski-Nagata theorem states that the $n$-th symbolic power $Q^{(n)}$ of $Q$ consists of the elements in $R$ that are taken inside $Q$ by every $K$-linear differential operator of order at most $n-1$.

%
%Zariski's Main Lemma on Holomorphic Functions \cite{Zariski} states that if $P$ is a prime ideal in a polynomial ring over a field, then
%$$P^{(n)} = \bigcap_{\m \in \mathcal{N}} \m^n,$$
%where the intersection is taken over the maximal ideals $\m$ containing $P$. For Nagata's proof, see \cite[p. 143]{Nagata}. This result was later refined by Eisenbud and Hochster \cite{EisenbudHochster}, and it can be rephrased using differential powers of ideals \cite{SurveySP}. More precisely, if $R$ is a smooth algebra over a perfect field $K$, and $Q$ is a prime ideal in $R$, the Zariski-Nagata Theorem states that the $n$-th symbolic power $Q^{(n)}$ of $Q$ consists of the elements in $R$ that are taken inside $Q$ by every $K$-linear differential operator of order at most $n-1$.

Rather than using perfectoid techniques, our generalization of Zariski-Nagata makes use of a different arithmetic notion of derivative, the notion of a $p$-derivation, defined by Joyal \cite{Joyal} and Buium \cite{Buium1995} independently. From a commutative algebra point of view, $p$-derivations are rather exotic maps from a ring to itself --- in particular, they are not even additive --- but they do have many applications to arithmetic geometry, such as in \cite{Buium-ADE,Buium-DCI,Borger}. To the best of our knowledge, this is the first application of $p$-derivations to commutative algebra.

While our results cover a more general setting, let us describe the case where $R = A \left[ x_1, \ldots, x_n \right]$, where $A$ denotes the integers $\mathbb{Z}$ or the $p$-adic integers $\mathbb{Z}_p$. Given a prime ideal $Q$ in $R$, we study two different types of differential powers associated to $Q$. The first one is defined just in terms of differential operators, as in the statement of the Zariski-Nagata theorem. More precisely, given an integer $n \geq 1$, the $n$-th ($A$-linear) differential power of $Q$ is defined as
\[
\ds Q \dif{n}{A} = \{ f\in R \ | \ \partial(f)\in Q \ \text{for all} \ \partial \in D^{n-1}_{R|A} \},
\]
where $D^{n-1}_{R|A}$ is the set of $A$-linear differential operators on $R$ of order at most $n-1$ (see Definition~\ref{def-diff-ops}). If $Q$ does not contain any prime integer, then $Q\dif{n}{A}$ coincides with the $n$-th symbolic power of $Q$.
\begin{theoremx} (see Theorem \ref{thm_dif=symb})
Let $R=A[x_1,\dots,x_n]$, where $A = \ZZ$ or $A = \ZZ_p$, and $Q$ be a prime ideal of $R$ such that $Q \cap A = (0)$. Then $Q^{(n)} = Q\dif{n}{A}$ for all $n \geq 1$. 
\end{theoremx}

More generally, the previous result holds if $R$ is an essentially smooth algebra over $A$, where $A$ is either $\ZZ$ or a DVR of mixed characteristic.

If the prime ideal $Q$ contains a prime integer $p$, then differential powers are not sufficient to characterize symbolic powers, as one can see in Remark~\ref{example_partial(p)=0}. To overcome this issue, we combine differential operators and $p$-derivations to define the mixed differential powers of an ideal. Given a fixed $p$-derivation $\delta$, the $n$-th mixed differential power of $Q$ is the ideal given by
\[
Q\difM{n}=\{ f \in R \ | \ (\delta^a \circ \partial)(f) \in Q  \ \text{for all} \ \partial \in D^b_{R|A} \ \text{with} \ a+b\leq n-1\}.
\]
In principle, the mixed differential powers of an ideal depend on the choice of a $p$-derivation $\delta$. However, in our setting $Q\difM{n}$ is independent of the choice of the $p$-derivation (see Corollary~\ref{mixed powers independent of delta}). This new notion of mixed differential powers allows us to characterize symbolic powers of prime ideals that contain a given integer $p$.

\begin{theoremx} (see Theorem \ref{thm_mdiff=symb})
Let $R=A[x_1,\dots,x_n]$, where $A = \ZZ$ or $A = \ZZ_p$, and $Q$ be a prime ideal of $R$ such that $Q \cap A = (p)$, for a prime $p$. Then $Q^{(n)} = Q\difM{n}$ for all $n \geq 1$.
\end{theoremx}

More generally, we show this holds for $R$ an essentially smooth algebra over $A$, where $A$ is either $\ZZ$ or a DVR with uniformizer $p$, as long as $R$ has a $p$-derivation and $A/pA$ satisfies some additional assumptions (e.g., $A/pA$ is perfect).

\section{Background}

\subsection{Essentially smooth algebras}

Throughout, we say that a ring $S$ is \emph{smooth} %(respectively, \emph{{\'e}tale}) 
over a subring $B$ if the inclusion map $B\rightarrow S$ is formally smooth %(respectively, {\'e}tale)  
and $S$ is a finitely generated $B$-algebra. We say that $S$ is \emph{essentially smooth} %(respectively, \emph{essentially {\'e}tale}) 
over $B$ if the inclusion map $B\rightarrow S$ is formally smooth %(respectively, {\'e}tale)
 and $S$ is a localization of a finitely generated $B$-algebra (i.e., $S$ is essentially of finite type over $B$). 

\begin{lemma}\label{free-basis}
	Let either 
	\begin{enumerate}
		\item\label{lem-case-1} $A=k$ be a field, or
		\item\label{lem-case-2} $A$ be $\ZZ$ or a DVR with uniformizer $p \in \ZZ$, and $k=A/pA$.
	\end{enumerate}
 Let $(R,\m,K)$ be a local ring that is essentially smooth over $A$, and suppose that $p\in \m$. Assume that the field extension $k \to K$ is separable. If $y_1,\dots,y_s$ is a minimal generating set of $\m$ in Case~(\ref{lem-case-1}) or $p,y_1,\dots,y_s$ is a minimal generating set of $\m$ in Case~(\ref{lem-case-2}), then there is a free basis for $\Omega_{R|A}$ that contains $dy_1,\dots,dy_s$.
\end{lemma}

\begin{proof} 
	By \cite[Theorem~25.2]{Matsumura}, there is a well-defined map $\varrho: \m / \m^2 \to \Omega_{R|A} \otimes_R K$ given by $[x] \mapsto dx\otimes 1$. Since, by the essential smoothness hypothesis, the $R$-module $\Omega_{R|A}$ is free, it suffices to show that $\varrho(y_1),\dots,\varrho(y_s)$ are $K$-linearly independent in $\Omega_{R|A}\otimes_R K$.
	
In Case~(\ref{lem-case-1}), this follows from \cite[Theorem~25.2]{Matsumura}, applied to the maps $k\to R \to R/\m$. 

In Case~(\ref{lem-case-2}), consider the commutative diagram given by
\[ \xymatrix{ \m / \m^2 \ar[r]^-{\varrho} \ar[d]_-\pi & \Omega_{R|A}\otimes_R K \ar[d]^-\cong \\ \overline{\m} / \overline{\m}^2 \ar[r]^-{\varrho} & \Omega_{\overline{R}|\overline{A}}\otimes_{\overline{R}} K, }\]
where the vertical maps are given by reduction modulo $p$, and $(\overline{R},\overline{\m},K)=(R/pR, \m/pR, K)$. By \cite[Theorem~25.2]{Matsumura}, the map $\varrho$ in the second row is injective. The kernel of the vertical map $\pi$ is generated by the class of $p$ modulo $\m^2$, and the vertical map on the right side is an isomorphism, since $p=0$ in $K$. It follows that the kernel of the map $\varrho$ in the first row is generated by the class of $p$ modulo $\m^2$, and the result follows.
	\end{proof}

\subsection{Differential operators} We now review some results regarding differential operators that we use in the rest of the paper. A general reference for differential operators is \cite[Chapter~16]{EGA-IV}; specific references to the facts we need are given below.

\begin{definition}\label{def-diff-ops}\cite[Section~16.8]{EGA-IV} Let $B\rightarrow S$ be a map of rings. The \emph{$B$-linear differential operators on $S$ of order $n$} are defined inductively as follows:
	\begin{itemize}
		\item $D^0_{S|B}=\Hom_S(S,S)\subseteq \Hom_B(S,S)$.
		\item $D^n_{S|B}=\big\{\delta\in \Hom_B(S,S) \ | \ [\delta,f]\in D^{n-1}_{S|B} \ \text{for all} \ f\in D^0_{S|B} \big\}$.
	\end{itemize}
\end{definition}

\begin{lemma}\label{EGA_existence_diff_operators}
	Let $B\rightarrow S$ be a formally smooth map of rings. Suppose that $\Omega_{S|B}$ is free, e.g., $S$ is local, and let $\{dz_{1}, \ldots, dz_{\Lambda} \}$ be a free basis for $\Omega_{S|B}$. Then there exists a family of differential operators $\{D_\alpha\}_{\alpha \in \NN^{\Lambda}}$ such that
\begin{itemize}
\item $D_\alpha(z^\beta)=\binom{\beta}{\alpha}z^{\beta-\alpha}$ for all $\beta\in \NN^{\Lambda}$ with $\beta_i\geq \alpha_i$ for all $i$, and 
\item $D_\alpha(z^\beta)=0$ for all $\beta\in \NN^{\Lambda}$ with $\beta_i < \alpha_i$ for some $i$.
\end{itemize}
The $S$-module $D^n_{S|B}$ is a free $S$-module generated by $\{D_\alpha: |\alpha|\leq n\}$ for each $n$.
\end{lemma}
\begin{proof} By \cite[Theorem~16.10.2]{EGA-IV}, $S$ is differentially smooth over $B$. Then, the statement above is the content of \cite[Theorem~16.11.2]{EGA-IV}.
	\end{proof}

\subsection{$p$-Derivations}
Fix a prime $p\in \ZZ$, and let $S$ be a ring on which $p$ is a nonzerodivisor.
The following operators were introduced independently in \cite{Joyal} and \cite{Buium1995}:

\begin{definition} \label{def_p-derivation} We say that a set-theoretic map $\delta:S \to S$ is a \emph{$p$-derivation} if $\phi_p(x) := x^p+p\delta(x)$ is a ring homomorphism. Equivalently, $\delta$ is a $p$-derivation if $\delta(1)=0$ and $\delta$ satisfies the following identities for all $x,y\in S$:
\begin{equation}
\label{prod}
\ds \delta(xy) = x^p\delta(y) + y^p \delta(x) + p\delta(x)\delta(y)
\end{equation}
and
\begin{equation}
\label{sum}
\ds \delta(x+y) = \delta(x) + \delta(y) + \mathcal{C}_p(x,y),
\end{equation}
where $\mathcal{C}_p(X,Y) = \frac{X^p+Y^p-(X+Y)^p}{p}\in \ZZ[X,Y]$.  
If $\delta$ is a $p$-derivation, we set $\delta^a$ to be the $a$-fold self-composition of $\delta$; in particular, $\delta^0$ is the identity. We set $\Der_p(S)$ to be the set of $p$-derivations on~$S$. For all positive integers $n$, we let
\[
\Der^{n}_p(S)={\{ \delta_1 \circ \cdots \circ \delta_t \ | \ \delta_i \in \Der_p(S) \ \text{for all $i$, and} \ t \leq n \}}.
\] 
\end{definition}

For a thorough development of the theory of $p$-derivations, see \cite{Buium-ADE}.

	Note that having a $p$-derivation on $S$ is equivalent to having a lift $\phi\!: S \rightarrow S$ of the Frobenius map $S/pS \rightarrow S/pS$. Indeed, it follows from the definition that if $\phi\!:S \rightarrow S$ is a map such that the induced map $\bar{\phi}:S/pS\rightarrow S/pS$ is the Frobenius map, then $\delta(x)=\frac{\phi(x)-x^p}{p}$ is a $p$-derivation. For example, if $R=\ZZ[x_1,\dots,x_n]$, then the map that sends a polynomial $f(x_1,\dots,x_n)$ to 
	$$\frac{f(x_1^p,\dots,x_n^p)-f(x_1,\dots,x_n)^p}{p}$$
	is a $p$-derivation.

However, not every ring admits a $p$-derivation. See \cite{Dupuy}, or consider the following example:

\begin{example} Let $S=\ZZ_p[x_1,\dots,x_n]$, and $R=S/(p-f)$, where $f\in (x_1,\dots,x_n)^2$. Suppose that there is some $p$-derivation $\delta$ on $R$. Then, considering $p=f$ as elements of $R$, by Proposition~\ref{prop_pdiff_power}~(\ref{prop_pdiff_power_1}) below, $\delta(f)\in (p,x_1,\dots,x_n)R$. However, by Remark~\ref{p-der-p-powers}, $\delta(p)\in (\ZZ \smallsetminus p \ZZ)$, which yields a contradiction.
\end{example} 

However, we do have the following:

\begin{proposition}\label{existence-p-der} A ring $S$ admits a $p$-derivation in each of the following cases:
	\begin{enumerate}
		\item\label{p-der-Z} $S=\ZZ$,
		\item\label{p-der-Witt} $S=W(C)$ is the Witt vectors over $C$ for some perfect ring of positive characteristic,
		\item\label{p-der-poly} $S$ is a polynomial ring over a ring $B$ that admits a $p$-derivation, or
		\item\label{p-der-complete} $S$ is $p$-adically complete and formally smooth over a ring $B$ that admits a $p$-derivation.

	\end{enumerate}
Suppose also that $\delta$ is a $p$-derivation on $S$ and $Q$ is a prime of $S$ containing $p$. Then there exists a $p$-derivation $\widetilde{\delta}$ on $S' \supseteq S$ such that $\widetilde{\delta}({s})={\delta(s)}$ for all $s\in S$ when:

	\begin{enumerate}[\indent (a)]
	\item \label{p-der-localize} $S'=S_Q$, or
	\item  \label{p-der-completion} $S'=\widehat{S_Q}$.
	\end{enumerate}
In Case~(\ref{p-der-localize}), we write $\widetilde{\delta}=\delta_Q$ and in Case~(\ref{p-der-completion}), we write $\widetilde{\delta}=\widehat{\delta_Q}$.
\end{proposition}
\begin{proof}
	As we have noted before, showing that a $p$-derivation exists is equivalent to proving that there exists a lift of Frobenius. Moreover, to verify that a $p$-derivation $\widetilde{\delta}$ extends $\delta$, it suffices to check that the associated lift of Frobenius extends the other. We verify for (1)--(4) that there is a lift of Frobenius.
	
	(\ref{p-der-Z}): The identity on $\ZZ$ is a lift of Frobenius.
	
	(\ref{p-der-Witt}): The Witt vectors admit a functorially induced Frobenius. 
		
	(\ref{p-der-poly}):  Extend a lift of Frobenius on $A$ by sending each variable to its $p$-th power.
	
	(\ref{p-der-complete}): Let $\alpha_1\!:S \rightarrow S/pS$ be the composition of the quotient map with the Frobenius map on $S/pS$. Since $S$ is formally smooth, there is a map $\alpha_2:S\rightarrow S/p^2 S$ such that $\alpha_1=\pi_2\circ\alpha_2$, where $\pi_2$ is the natural surjection of $S/p^2 S \to S/p S$. Inductively, by formal smoothness one obtains a family of maps $\alpha_i:S\rightarrow S/p^i S$ such that $\pi_i \circ \alpha_i = \alpha_{i-1}$. This compatible system of maps induces a map $\varprojlim \alpha_i: S\rightarrow \varprojlim S/p^i S \cong S$ that is a lift of the Frobenius.
	
	(\ref{p-der-localize}): Let $\Phi:S\to S$ be a lift of the Frobenius. We note that if $\Phi(t)\in Q$, then, since $\Phi(t)-t^p\in pS$ and $p\in Q$, we also have $t^p\in Q$, and hence $t\in Q$. It follows that $\Phi$ induces a map $\Phi_Q:S_Q \to S_Q$. Now, we claim that $\Phi_Q$ is a lift of the Frobenius as well. In fact, if $s/t\in S_Q$, observe that $\Phi_Q(s/t)-s^p/t^p \in p S_Q$. To see this, note that
	\[\Phi_Q\left(\frac{s}{t}\right)-\frac{s^p}{t^p} = \frac{\Phi(s) t^p - \Phi(t) s^p}{t^p \Phi(t)} = \frac{t^p(\Phi(s) - s^p) - s^p(\Phi(t)-t^p)}{t^p \Phi(t)}, \]
	where the numerator is a multiple of $p$, and the denominator is not.
%Let $\Phi:S\to S$ be a lift of the Frobenius. We claim that $\Phi_Q:S_Q \to S_Q$ is a lift of the Frobenius as well. We note first that if $\Phi(t)\in Q$, then, since $\Phi(t)-t^p\in pS$ and $p\in Q$, we also have $t^p\in Q$, and hence $t\in Q$. It follows that $\Phi_Q$ does indeed map $S_Q$ to $S_Q$. Now, if $s/t\in S_Q$, observe that $\Phi_Q(s/t)-s^p/t^p \in p S_Q$. Indeed,
%\[\Phi_Q\left(\frac{s}{t}\right)-\frac{s^p}{t^p} = \frac{\Phi(s) t^p - \Phi(t) s^p}{t^p \Phi(t)} = \frac{t^p(\Phi(s) - s^p) - s^p(\Phi(t)-t^p)}{t^p \Phi(t)}, \]
%where the numerator is a multiple of $p$, and the denominator is not.
	
	(\ref{p-der-completion}): Given a lift $\Phi$ of the Frobenius on $S$, to see that it extends to $\widehat{S_Q}$ it suffices to check that $\Phi(Q)\subseteq Q$. In fact, in this case, we have $\Phi(Q^n) \subseteq \Phi(Q)^n \subseteq Q^n$ for all positive integers $n$, since $\Phi$ is a ring homomorphism, and it follows that $\Phi$ is $Q$-adically continuous. To see that $\Phi(Q) \subseteq Q$, observe that, for $s\in Q$, $\Phi(s)\equiv s^p \ \mathrm{mod} \, pS$, so $\Phi(s)\in Q$, because $p \in Q$ by assumption.
\end{proof}

\begin{remark}\label{p-der-p-powers}
	Repeated application of Equation~\eqref{sum} shows that a $p$-derivation sends the prime ring of $R$ (i.e., the canonical image of $\ZZ$) to itself. If $R$ has characteristic zero, so that its prime ring is $\ZZ$, any $p$-derivation on $R$ restricts to a $p$-derivation on $\ZZ$. On the other hand, there is a unique $p$-derivation on $\ZZ$ given by the \emph{Fermat difference operator}: $\delta(n)=\frac{n-n^p}{p}$. In particular, when $R$ has characteristic zero, every $p$-derivation satisfies $\delta(p^n)=p^{n-1}-p^{pn-1} \in (p)^{n-1} \smallsetminus (p)^n$ for all $n \geqslant 1$.
	
	%	Moreover, if $\delta$ is a $p$-derivation on $R$, 
	%	
	%	$\delta(p^t)\in (p^{t-1}\ZZ\smallsetminus p^t\ZZ) \subseteq (p^{t-1}R \smallsetminus p^t R)$. The base case $t=1$ for induction is the previous claim. Then,
	%	\[\delta(p^t)=\delta(p \cdot p^{t-1})=p^p \delta(p^{t-1})+p^{pt-p} \delta(p) + p \delta(p)\delta(p^{t-1}).\]
	%	Applying the induction hypothesis, the $p$-adic orders of the terms are $t+p-2$, $t(p-1)$, and $t-1$, so the $p$-adic order of the sum is $t-1$.
\end{remark}

\begin{remark}\label{Remark_Cp} Let $R$ be a ring, and $I,J \subseteq R$ be ideals. Let $\delta$ be a $p$-derivation. If $a \in I$ and $b \in J$, then $\delta(a+b) \equiv \delta(a) + \delta(b)$ modulo $IJ$. In fact, we have that
	\[
	\ds \mathcal{C}_p(a,b) = \frac{a^p+b^p-(a+b)^p}{p} = \sum_{i=1}^{p-1} \frac{{p \choose i}}{p} a^ib^{p-i} \in IJ,
	\]
	because $p \mid {p \choose i}$ for all $1 \leq i \leq p-1$. In particular, we have that $\mathcal{C}_p(a,b) \in (a) \cap (b)$. With similar considerations, one can show that if $a,b \in I$, then $\delta(a+b) \equiv \delta(a)+\delta(b)$ modulo $I^p$.
\end{remark}

\section{Results}
\subsection{Primes not containing $p$} In this subsection, we focus on differential and symbolic powers of prime ideals that do not contain any prime integer. To study symbolic powers of such ideals, we use differential operators.

\begin{definition}\cite{SurveySP}\label{def_dif} Let $S$ be a ring, $B$ be a subring of $S$, and $I$ be an ideal of $S$. The \emph{$n$-th ($B$-linear) differential power of $I$} is
\[ I\dif{n}{B} := \{ f\in S \ | \ \partial(f)\in I \ \text{for all} \ \partial \in D^{n-1}_{S|B} \}.\] 
\end{definition}

The following proposition is a generalization of \cite[Proposition~2.4]{SurveySP}.

\begin{proposition}  \label{prop_ordinary_dif}
Let $S$ be a ring, $B$ be a subring of $S$, and $\a$ an ideal of $S$. The following properties hold:
\begin{enumerate}
	\item \label{prop_ordinary_dif_1} $\a\dif{n}{B}$ is an ideal, and $\a\dif{n}{B} \subseteq \a$.
\item  \label{prop_ordinary_dif_1.5} $\a\dif{n+1}{B} \subseteq \a\dif{n}{B}$ for all $n$.
	\item \label{prop_ordinary_dif_2} For any $0 \leq t \leq n$, and any $\partial \in D^t_{S|B}$, we have $\partial(\a^n) \subseteq \a^{n-t}$. In particular, $\a^n \subseteq \a\dif{n}{B}$.
	\item \label{prop_ordinary_dif_3} If $Q$ is a prime ideal of $S$, and $\a$ is $Q$-primary, then $\a\dif{n}{B}$ is $Q$-primary.
	\item \label{prop_ordinary_dif_4} If $Q$ is a prime ideal of $S$, and $\a$ is $Q$-primary, then $\a^{(n)} \subseteq \a\dif{n}{B}$.
\end{enumerate}
\end{proposition}
\begin{proof}
The proof of (\ref{prop_ordinary_dif_1}) and (\ref{prop_ordinary_dif_1.5}) is analogous to that of \cite[Proposition 2.4]{SurveySP}, where the same claim is made for the case when $B$ is a field. 

For part~(\ref{prop_ordinary_dif_2}), we first proceed by induction on $t \geq 0$. If $t=0$, then $\partial \in D^0_{S|B}$ is just multiplication by an element of $S$, and the statement is clear. If $t\geq 1$, we proceed by induction on $n -t \geq 0$. If $n=t$, then $\partial(\a^n) \subseteq \a^0 = S$ is trivial. By induction, assume that $\partial(\a^{n-1}) \subseteq \a^{n-1-t}$. To conclude the proof that $\partial(\a^n)\subseteq \a^{n-t}$, it suffices to show that $\partial(xy) \in \a^{n-t}$ for all $x \in \a$ and $y \in \a^{n-1}$. To see this, observe that $\partial(xy) = [\partial,x](y) + x \partial(y)$, with $[\partial,x] \in D^{t-1}_{S|B}$. By the inductive hypothesis on $t$, we have that $[\partial,x](y) \in \a^{n - 1-(t-1)} = \a^{n-t}$. Since $x \in \a$ and $y \in \a^{n-1}$, we also have $x \partial(y) \in \a\a^{n-1-t} = \a^{n-t}$, and the claim follows. In particular, this shows that $\partial(\a^n) \subseteq \a$ for all $\partial$ of order up to $n-1$, so that $\a^n \subseteq \a\dif{n}{B}$.

Part~(\ref{prop_ordinary_dif_4}) follows from (\ref{prop_ordinary_dif_3}), given that $\a^n \subseteq \a\dif{n}{B}$, and that $\a^{(n)}$ is the smallest $Q$-primary ideal that contains $\a^n$. 

To show (\ref{prop_ordinary_dif_3}), we first observe that $\a^n \subseteq \a\dif{n}{B} \subseteq \a$ by parts~(\ref{prop_ordinary_dif_1})~and~(\ref{prop_ordinary_dif_2}), so that $\sqrt{\a\dif{n}{B}}=Q$. To prove that $\a\dif{n}{B}$ is $Q$-primary, we proceed by induction on $n \geq 1$. The case $n=1$ is true by assumption, since $\a\dif{1}{B} = \a$. Let $xy \in \a\dif{n}{B}$, with $x \notin Q$. Observe that $xy \in \a\dif{n}{B} \subseteq \a\dif{n-1}{B}$ by part (\ref{prop_ordinary_dif_1.5}), and by the inductive hypothesis the latter ideal is $Q$-primary. Since $x \notin Q$, it follows that $y \in \a\dif{n-1}{B}$. Let $\partial \in D^{n-1}_{S|B}$, so that $[\partial,x] \in D^{n-2}_{S|B}$. It follows that $[\partial,x](y) \in \a$, by definition of $\a\dif{n-1}{B}$. On the other hand, we also have $\partial (xy) \in \a$, and thus $x \partial(y) = \partial(xy) - [\partial,x](y)  \in \a$. Using again that $x \notin Q$, and that $\a$ is $Q$-primary, it follows that $\partial(y) \in \a$. Since $\partial \in D^{n-1}_{S|B}$ was arbitrary, we conclude that $y \in \a\dif{n}{B}$, and thus $\a\dif{n}{B}$ is $Q$-primary.
\end{proof}

\begin{corollary}\label{fixing the mix up}
	In the context of Lemma~\ref{EGA_existence_diff_operators}, fix $\alpha$ such that $\alpha_{t+1} = \cdots = \alpha_\Lambda = 0$, for some $t \leqslant \Lambda$. Consider the ideal $I = \left( z_1, \ldots, z_t \right)$, and let $u \notin I$. Then:
	\begin{enumerate}
		\item $D_{\alpha} \left( u z^{\alpha} \right) \notin I$, and
		\item For all $\beta \neq \alpha$ with $\left| \beta \right| = \left| \alpha \right|$, $D_{\alpha} \left( u z^{\beta} \right) \in I$.
	\end{enumerate}
\end{corollary}

\begin{proof}
	First, note that $z^{\beta} \in I^{\left| \beta \right|}$. By Proposition~\ref{prop_ordinary_dif}~(\ref{prop_ordinary_dif_2}), $\partial \left( z^{\beta} \right) \in I$ for all $\partial \in D^{|\beta|-1 }_{S|B}$. Therefore, every differential operator $\left[ D_{\alpha}, u \right]$ with $|\alpha| = |\beta|$, takes $z^{\beta}$ into $I$. It follows that
	\begin{enumerate}
		\item $D_{\alpha} \left( u z^{\alpha} \right) = \left[ D_{\alpha}, u \right] \left( z^{\alpha} \right) + u D_{\alpha} \left( z^{\alpha} \right) =\left[ D_{\alpha}, u \right] \left( z^{\alpha} \right) + u \notin I$, and
		\item $D_{\alpha} \left( u z^{\beta} \right) = \left[ D_{\alpha}, u \right] \left( z^{\beta} \right) + u D_{\alpha} \left( z^{\beta} \right) =\left[ D_{\alpha}, u \right] \left( z^{\beta} \right) \in I$.\qedhere
	\end{enumerate}
\end{proof}
 
 We will need the following lemma on the behavior of differential powers under localization. The following is from forthcoming work by Brenner, N\'{u}\~{n}ez-Betancourt, and the third author \cite{DiffSig}. We include a proof here for completeness, while we refer the reader to \cite{DiffSig} for a thorough treatment and other applications of differential powers. We thank Holger Brenner and Luis N\'{u}\~{n}ez-Betancourt for allowing us to share this result here.
 
 \begin{lemma}[\cite{DiffSig}]\label{diff-powers-localize} Let $S$ be a ring, $B$ be a subring of $S$, $W$ be a multiplicatively closed subset of $S$, $V=W\cap B$, and $I$ be an ideal of $S$ such that $W^{-1} I \cap S=I$. Suppose that $S$ is essentially of finite type over $B$. Then
 	\begin{enumerate}
 		\item\label{diff-loc-1} $(W^{-1}I)\dif{n}{B}=(W^{-1}I)\dif{n}{V^{-1}B}$,
 		\item\label{diff-loc-2} $I\dif{n}{B}=(W^{-1}I)\dif{n}{B} \cap S$, and
 		\item\label{diff-loc-3} $W^{-1}I\dif{n}{B}=(W^{-1}I)\dif{n}{B}$.
 	\end{enumerate} 
 \end{lemma}
 
 We first record the following lemma, which is well-known in the case $B$ is a field.
 
 \begin{lemma}\cite{DiffSig} \label{lemma_Grothendieck}
 	With notation as above, there are isomorphisms
 	\[ W^{-1}D^n_{S|B} \cong D^n_{W^{-1}S | B} \cong D^n_{W^{-1}S | V^{-1}B}.  \]
 	In particular, every $\delta\in D^n_{S|B}$ extends to an element $\widetilde{\delta}\in D^n_{W^{-1}S | V^{-1}B}$.
 \end{lemma}
 \begin{proof}
 	By \cite[16.8.1]{EGA-IV} and \cite[16.8.8]{EGA-IV}, there are isomorphisms $D^n_{T|C}\cong \Hom_T(P^n_{T|C},T)$ for all algebras $C\to T$, where $P^n_{T|C}$ denotes the module of principal parts. By \cite[16.4.22]{EGA-IV}, each $P^n_{S|B}$ is a finitely generated $S$-module. By \cite[16.4.15.1]{EGA-IV}, there are isomorphisms $W^{-1}P^n_{S|B}\cong P^n_{W^{-1}S|V^{-1}B}$, and by \cite[16.4.14.1]{EGA-IV} these modules are isomorphic to $P^n_{W^{-1}S|B}$. We caution the reader that the proof of \cite[16.4.14]{EGA-IV} contains an error, but the statements are correct. The stated isomorphisms now follow.
 \end{proof}
 
 \begin{proof}[Proof of Lemma~\ref{diff-powers-localize}]
 	
Part~(\ref{diff-loc-1}) is immediate from the previous lemma.

We prove part~(\ref{diff-loc-2}). In order to show that $I\dif{n}{B}\subseteq(W^{-1}I)\dif{n}{B} \cap S$, it suffices to prove that if $D^{n-1}_{S|B}(s)\subseteq I$ for some $s \in S$, then $D^{n-1}_{W^{-1}S | B}(\frac{s}{1}) \subseteq W^{-1}I$. For any $\partial\in D^{n-1}_{W^{-1}S | B}$, by Lemma \ref{lemma_Grothendieck} there exists $w\in W$ and $\eta\in D^{n-1}_{S|B}$ such that $\partial(\frac{s}{1})= \frac{\eta(s)}{w}$ for all $s\in S$. The claim is then clear. For the other containment, suppose that $s\in S$, and $D^{n-1}_{W^{-1}S | B}(\frac{s}{1}) \subseteq W^{-1}I$. If $\partial\in D^{n-1}_{S|B}$, then it extends to a differential operator $\widetilde{\partial}\in D^{n-1}_{W^{-1}S | B}$ such that $\widetilde{\partial}(\frac{s}{1})=\frac{\partial(s)}{1}$. By hypothesis, this element is in $W^{-1}I \cap S =I$. Thus, $s$ lies in $I\dif{n}{B}$.

We now prove part (\ref{diff-loc-3}). To show that $W^{-1}I\dif{n}{B}\subseteq(W^{-1}I)\dif{n}{B}$, we proceed  by induction on $n$. The case $n=1$ is trivial. Let $s\in I\dif{n}{B}$, $w\in W$, and $\partial\in D^{n-1}_{W^{-1}S | B}$. Then, $\partial(\frac{s}{w})=\frac{1}{w}(\partial(s)-[\partial,w](\frac{s}{w}))$. By the induction hypothesis, $[\partial,t](\frac{s}{t})\in W^{-1}I$. By Lemma \ref{lemma_Grothendieck}, there exists $t\in W$ and $\eta\in D^{n-1}_{S|B}$ such that $\partial(\frac{s}{1})= \frac{\eta(s)}{t}$. Since $\eta(s) \in I$ by hypothesis, we also have that $\partial(s)\in W^{-1}I$. We now prove the containment $W^{-1}I\dif{n}{B}\supseteq(W^{-1}I)\dif{n}{B}$. Since elements of $D^{n-1}_{S|B}$ extend to elements of $D^{n-1}_{W^{-1}S | B}$, we know that $D^{n-1}_{S|B}(s) \subseteq D^{n-1}_{W^{-1}S | B}(s)$ for all $s\in S$. But then 
 	\[D^{n-1}_{S|B}(s) \subseteq D^{n-1}_{W^{-1}S | B} (s) = D^{n-1}_{W^{-1}S | B} \left(\frac{st}{t}\right) = \left(D^{n-1}_{W^{-1}S | B} \cdot t\right) \left(\frac{s}{t}\right) \subseteq D^{n-1}_{W^{-1}S | B}  \left(\frac{s}{t}\right) ,\] since multiplication by an element does not increase the order of a differential operator. By hypothesis,  $D^{n-1}_{W^{-1}S | B}\left(\frac{s}{t}\right) \subseteq W^{-1} I$, so $D^{n-1}_{S|B} (s)\subseteq W^{-1}I \cap S =I$, as required.
\end{proof}

 As noted in the introduction, the Zariski-Nagata theorem can be stated in terms of differential powers of ideals  \cite{SurveySP}. Namely, if $S=K[x_1,\dots,x_d]$, $K$ is a perfect field, and $Q$ is a prime ideal of $S$, then $Q\dif{n}{K}=Q^{(n)}$ for all $n$. 
We can give a concise proof of the Zariski-Nagata theorem for algebras over fields by combining the previous two results.

\begin{theorem}[Zariski-Nagata]\label{ZN-field}
	Let $K$ be a field, $R$ be essentially smooth over $K$, and $Q$ be a prime ideal of $R$. If $K\hookrightarrow R_Q / Q R_Q$ is separable, then $Q\dif{n}{K}=Q^{(n)}$ for all $n$. In particular, if $K$ is perfect, then $Q\dif{n}{K}=Q^{(n)}$ for every prime $Q$ and all $n$.
\end{theorem}

\begin{proof}
%	First, we note that $Q^{(n)} \subseteq Q\dif{n}{K}$ by Proposition~\ref{prop_ordinary_dif}~(\ref{prop_ordinary_dif_4}), so it suffices to prove $Q\dif{n}{K} \subseteq Q^{(n)}$. Since $Q$ is the only associated prime of $Q^{(n)}$, it is sufficient check the equality $Q^{(n)}R_Q = Q\dif{n}{K}R_Q$. By Lemma~\ref{diff-powers-localize}, $Q\dif{n}{K}R_Q=(Q R_Q)\dif{n}{K}$, and since $Q$ is maximal in $R_Q$, $Q^{(n)}R_Q = Q^n R_Q$. Thus, it suffices to show that $(Q R_Q)\dif{n}{K} \subseteq Q^n R_Q$.
It suffices to check the equality $Q^{(n)}R_Q = Q\dif{n}{K}R_Q$. By Lemma \ref{diff-powers-localize}, $(QR_Q)\dif{n}{K}R_Q=Q\dif{n}{K}R_Q$, and since $Q$ is maximal in $R_Q$, we have that $Q^{(n)}R_Q =  Q^nR_Q$. Thus, it suffices to show $(QR_Q)\dif{n}{K}R_Q = Q^nR_Q$.
	
	Let $y_1,\dots,y_t$ be a minimal generating set for $Q R_Q$. Suppose that there exists an element $f \in (Q R_Q)\dif{n}{K}$ of order $s < n$, meaning that $f \in Q^s R_Q$, $f \notin Q^{s+1} R_Q$. Then, we can write $f= \sum_i u_i y^{\alpha_i} + g$ for some units $u_i \in R_Q$, some $\alpha_i$ with $|\alpha_i| = s$, and some $g \in Q^{s+1}R_Q$. Fix some multi-index $\alpha \in \NN^d$ with $| \alpha | = s$ and some unit $u$ such that $u y^\alpha$ appears in the expression of $f$ as above. By Lemma~\ref{free-basis}, $dy_1,\dots,dy_t$ form part of a free basis for $\Omega_{R|K}$. Thus, by Theorem~\ref{EGA_existence_diff_operators} and Corollary~\ref{fixing the mix up}, there exists a differential operator $D_\alpha \in D^{s}_{R_Q|K}$ such that $D_{\alpha} \left( u y^{\alpha} \right) \notin Q R_Q$ and $D_{\alpha} \left( u_i y^{\alpha_i} \right) \in Q R_Q$ for each other term $u_i y^{\alpha_i}$ in the expression for $f$. Additionally, since $g\in Q^{s+1}R_Q$ and $D_\alpha \in D^{s}_{R_Q|K}$, we have $D_\alpha(g)\in Q$ by Proposition \ref{prop_ordinary_dif} (\ref{prop_ordinary_dif_2}). It follows that $D_\alpha(f)\notin Q R_Q$, contradicting the assumption $f \in (Q R_Q)\dif{n}{K}$.
\end{proof}

We note that the essential smoothness hypothesis is necessary.

\begin{example}
	Let $K$ be a field, $R=K[x,y,z]/(y^2-xz)$, $Q=(x,y)$, and $\m=(x,y,z)$. Then, $x\in Q^{(2)} \smallsetminus \m^2$. However, it is evident from Definition~\ref{def_dif} that $Q\dif{2}{K}\subseteq \m\dif{2}{K}$. Thus, the conclusion of Theorem~\ref{ZN-field} cannot hold.
\end{example}

 The following example shows that the conclusion of the Zariski-Nagata theorem may fail if the field extension $K \to R_Q/QR_Q$ 
 is not separable.

\begin{example} Let $K=\FF_p(t)$, $R=K[x]$, and $Q=(x^p-t)$. We claim that $Q\dif{2}{K}=Q$. By Lemma~\ref{EGA_existence_diff_operators}, $D^1_{R|K}=R \oplus R \frac{d}{dx}$ (where $\frac{d}{dx}=D_1$ in the notation of Lemma~\ref{EGA_existence_diff_operators}),  so it suffices to show that $\frac{d}{dx}(x^p-t)\in Q$. Indeed, $\frac{d}{dx}(x^p-t)=0$, so the claim is verified. Since $Q$ is a principal ideal, $Q^{(2)}=Q^2$. In particular, $Q^{(2)}\neq Q\dif{2}{K}$.
\end{example}

We are now ready to state our first main result: a version of the Zariski-Nagata theorem for prime ideals that do not contain any prime integer.
\begin{theorem} \label{thm_dif=symb} Let $A$ be either $\ZZ$ or a DVR of mixed characteristic. Let $R$ be an essentially smooth $A$-algebra. If $Q \in \Spec(R)$ is such that $Q \cap A = (0)$, then $Q^{(n)}=Q\dif{n}{A}$.
\end{theorem}
\begin{proof} Let $A'=A\otimes_{\ZZ}\QQ$ and $R'=R\otimes_A A'=R\otimes_{\ZZ}\QQ$. We note that $A'$ is a field of characteristic zero, and $R'$ is formally smooth and essentially of finite type over $A'$. Observe that $QR'$ is a prime ideal in $R'$. We claim that $(QR')^{(n)}\cap R = Q^{(n)}$. Indeed, $A\cap Q=(0)$, so $R_Q \cong R'_{QR'}$, and thus 
	\[ Q^{(n)} = Q^n R_Q \cap R = ((QR')^n R'_{QR'} \cap R') \cap R = (QR')^{(n)}\cap R.\]
	
	We note next that $(QR')^{(n)}=(QR')\dif{n}{A'}$; indeed, Theorem~\ref{ZN-field} applies. Finally, applying parts (1) and (2) of Lemma~\ref{diff-powers-localize}, we conclude that
	\[ Q^{(n)} = (QR')^{(n)}\cap R = (QR')\dif{n}{A'}\cap R = Q\dif{n}{A}. \qedhere \]
\end{proof}
As an application of Theorem~\ref{thm_dif=symb}, we obtain a generalization of Zariski's main lemma on holomorphic functions \cite{Zariski,EisenbudHochster}.
\begin{corollary} Let $A$ be as in Theorem~\ref{thm_dif=symb}, and assume that $R$ is smooth over $A$. Let $A' = A \otimes_\ZZ \QQ$, and $R' = R\otimes_A A'$. For a prime $Q\subseteq R$ not containing any prime integer, set $\mathcal{B} = \{\n \cap R \mid \n \in \Max\Spec(R') \cap \V(QR')\}$. We have $Q^{(n)} = \bigcap_{\q \in \mathcal{B}} \q^{(n)}$. 
\end{corollary}
\begin{proof}
By Theorem \ref{thm_dif=symb}, it suffices to show that $Q\dif{n}{A} = \bigcap_{\q \in \mathcal{B}} \q\dif{n}{A}$. Since $Q \subseteq \q$ for all $\q \in \mathcal{B}$, it follows that $Q\dif{n}{A} \subseteq \bigcap_{\q \in \mathcal{B}} \q\dif{n}{A}$. For the converse, we claim that $Q = \bigcap_{\q \in \mathcal{B}} \q$.  Let $J=\bigcap_{\q \in \mathcal{B}} \q$, and let $p$ be a prime integer. Note that $J \subseteq \bigcap_{\q \in \mathcal{B}} (\q + (p)) = Q + (p)$, where the equality follows from the fact that $R/pR$ is a Hilbert-Jacobson ring, and that $\{\q+(p) \mid \q \in \mathcal{B}\} = \Max\Spec(R) \cap \V(Q+(p))$. Observe that $p$ is a nonzerodivisor on $R/J$, because $p \notin \q$ for all $\q \in \mathcal{B}$. Let $a \in J$. It follows from the containments proved above that we can write $a=x+py$, where $x \in Q$ and $y \in R$. Since $Q \subseteq J$, we conclude that $py = a-x \in J$, which implies that $y \in J$. Therefore, we have that $J \subseteq Q + pJ \subseteq  Q+\m J \subseteq J$ for every maximal ideal $\m$ of $R$ that contains $p$. Since the prime integer $p$ was arbitrary, this argument shows that $J = Q + \m J$ for every maximal ideal $\m$ of $R$. Nakayama's lemma applied to the localization at each such ideal allows us to conclude that $J=Q$, as desired. Now let $f \in \bigcap_{\q \in \mathcal{B}} \q\dif{n}{A}$, and $\delta \in D^{n-1}_{R|A}$ be arbitrary. By assumption, we have that $\delta(f) \in \bigcap_{\q \in \mathcal{B}} \q$. Finally, our previous claim shows that $\delta(f) \in Q$, and we conclude that $f \in Q\dif{n}{A}$. 
\end{proof}

\subsection{Primes containing $p$} Throughout this subsection, $p$ is a prime integer, and we assume that all rings involved are $p$-torsion free; we note that this condition is implied by the hypotheses of Setting~\ref{setting-containsP} below. We now move our attention to prime ideals that contain a given prime integer $p$. 
The following remark shows that, in order to characterize symbolic  powers of such ideals, it is not sufficient to rely just on differential operators.

\begin{remark}\label{example_partial(p)=0}
	Let $S$ be a ring of characteristic zero, $B\subseteq S$ a subring, and $Q$ be a prime ideal of $S$ such that $Q\cap \ZZ=(p)\neq (0)$. Then, since every $B$-linear differential operator $\partial$ on $S$ satisfies $\partial(p) = p \partial(1) \in (p) \subseteq Q$, we have $p \in Q\dif{n}{B}$ for all $n$. Because $\bigcap_{n\in\NN} Q^{(n)} = 0$, $Q\dif{n}{B} \neq Q^{(n)}$ for all but finitely many $n$.

For a simple concrete example, let $S=\ZZ$ and $Q=(2)$. Then $Q^n=(2^n)$ for all $n$, whereas $Q\dif{n}{\ZZ}=(2)$ for all $n$.
\end{remark}

Remark~\ref{example_partial(p)=0} shows that differential powers are too large for our purpose. The issue with this approach is that differential operators cannot decrease the $p$-adic order of an element. On the other hand, $p$-derivations have this feature, and this motivates the following definition.

\begin{definition} \label{defn_derp} Let $p\in \ZZ$ be a prime, $S$ be a ring with a $p$-derivation $\delta$, and $I$ be an ideal of $S$. The \emph{$n$-th $p$-differential power of $I$ with respect to $\delta$} is
\[ I\derp{n}:=\{ f\in S \ | \ \delta^a(f) \in I \ \text{for all} \ a\leq n-1 \}.\]
We note that this definition depends on the choice of $\delta$, although this is suppressed to avoid conflicting notation with other notions. 
\end{definition}
It immediately follows from Definition \ref{defn_derp} that $I\derp{n+1} \subseteq I\derp{n}$ for all $n$. In what follows, we will be mainly concerned with $p$-differential powers $I\derp{n}$ with respect to a fixed $p$-derivation $\delta$. However, it is convenient to define a similar object, where we take into account all $p$-derivations on $S$ at the same time:
\[
\ds I^{\lr{n}_{\Der_p}} := \{ f\in S \ | \ \Der_p^{a}(f) \subseteq I \ \text{for all} \ a\leq n-1 \}.
\]
\begin{proposition}\label{prop_pdiff_power} Let $p\in \ZZ$ be a prime, $S$ be a ring with a $p$-derivation $\delta$, and $I$ be an ideal of $S$. If $I$ is an ideal, then $I\derp{n}$ and $I^{\lr{n}_{\Der_p}}$ are ideals for all $n$. Moreover, if  $Q \in \Spec(S)$ is a prime ideal, and $\a$ is a $Q$-primary ideal that contains $p$, we have the following properties:
\begin{enumerate}
\item \label{prop_pdiff_power_2b} For any $s,t \in \NN$ we have $\a\derp{s}\a\derp{t}\subseteq \a\derp{s+t}$ and $\a^{\lr{s}_{\Der_p}}\a^{\lr{t}_{\Der_p}}\subseteq \a^{\lr{s+t}_{\Der_p}}$.
\item \label{prop_pdiff_power_1} $\delta(\a^n) \subseteq \a^{n-1}$ for all $n$. In particular, $\a^n \subseteq \a\derp{n}$ and $\a^n \subseteq \a^{\lr{n}_{\Der_p}}$.
\item \label{prop_pdiff_power_3} $\a\derp{n}$ and $\a^{\lr{n}_{\Der_p}}$ are $Q$-primary ideals.
\item \label{prop_pdiff_power_4} $\a^{(n)} \subseteq \a\derp{n}$ and $\a^{(n)} \subseteq \a^{\lr{n}_{\Der_p}}$.
\end{enumerate}

\end{proposition}
\begin{proof}
We only prove the statements for $p$-differential powers with respect to the given $p$-derivation $\delta$, since the proofs for $I^{\lr{n}_{\Der_p}}$ and $\a^{\lr{n}_{\Der_p}}$ are completely analogous. 

We first show by induction on $n$ that $I\derp{n}$ is an ideal. For $n=1$ we have $I\derp{1}=I$, and the statement is trivial. Assume that $I\derp{n-1}$ is an ideal, and let $x,y\in I\derp{n}$. By induction, we have that $\delta(x)+\delta(y) \in I\derp{n-1}$, since $\delta(x),\delta(y) \in I\derp{n-1}$. It follows that $\delta(x+y) = \delta(x)+\delta(y) + \mathcal{C}_p(x,y) \in I\derp{n-1}$, because $\mathcal{C}_p(x,y) \in (x) \subseteq I\derp{n} \subseteq I\derp{n-1}$ by Remark \ref{Remark_Cp}. This shows that $x+y \in I\derp{n}$. Now let $x \in I\derp{n}$ and $b \in S$. By induction, we may assume that $yc \in I\derp{n-1}$ for all $y \in I\derp{n-1}$ and all $c \in S$. Since $\delta(x) \in I\derp{n-1}$, it follows that $\delta(xb) = x^p\delta(b) + b^p\delta(x) + p\delta(b)\delta(x) \in I\derp{n-1}$, because $x \in I\derp{n}\subseteq I\derp{n-1}$. This finishes the proof that $I\derp{n}$ is an ideal. 

Now we let $\a$ be a $Q$-primary ideal of $S$ that contains $p$.

(\ref{prop_pdiff_power_2b}) We proceed by induction on $s+t \geq  0$. If either $s=0$, or $t=0$, then the claim is trivial. This takes care of the base case for the induction, and allows us to assume henceforth that $s \geq 1$ and $t \geq 1$. Let $x \in \a\derp{s}$ and $y \in \a\derp{t}$. We observe that
\[
\ds \delta(xy) = x^p\delta(y) + y^p \delta(x) + p \delta(x)\delta(y) \in \a\derp{s}\a\derp{t-1} + \a\derp{t} \a\derp{s-1} + \a\cdot \a\derp{s-1}\a\derp{t-1},
\]
and by induction we obtain that $\delta(xy) \in \a\derp{s+t-1}$. This shows that $xy \in \a\derp{s+t}$, as claimed.

(\ref{prop_pdiff_power_1}) We proceed by induction on $n\geq 1$. For $n=1$ the statements are clear. Assume that $\delta(\a^n) \subseteq \a^{n-1}$ holds true; we want to show that $\a^{n+1}$ is mapped inside $\a^n$ by $\delta$. We first show that any element of the form $xy$, with $x\in \a$ and $y \in \a^n$ satisfies $\delta(xy) \in \a^n$. Using the inductive hypothesis, we get
\[
\ds \delta(xy) = x^p\delta(y) + y^p \delta(x) + p\delta(x)\delta(y) \in \a^p \a^{n-1} + \a^{pn} + p \cdot \a^{n-1}.
\]
Given that $p \in \a$, we conclude that $\delta(xy) \in \a^n$, as desired. Because every element of $\a^{n+1}$ can be written as a sum of elements $xy$, with $x \in \a$ and $y \in \a^n$, it suffices to show that the sum of any two such elements is mapped by $\delta$ inside $\a^n$. Let $z,w \in \a^{n+1}$ be elements of this form, and recall that $\delta(z+w) = \delta(z) + \delta(w) + \mathcal{C}_p(z,w)$. Since $\mathcal{C}_p(z,w) \in (z) \subseteq \a^{n+1}$, and $\delta(z),\delta(w) \in \a^n$ by what we have shown above, we conclude that $\delta(z+w) \in \a^n$. The final claim now follows immediately, since $\delta(\a^{n+1}) \subseteq \a^n \subseteq \a\derp{n}$ by induction, so that $\a^{n+1} \subseteq \a\derp{n+1}$.

(\ref{prop_pdiff_power_3}) Since $\a^n \subseteq \a\derp{n} \subseteq \a$, it follows immediately that $\sqrt{\a\derp{n}} = Q$ for all $n$. To show that $ \a\derp{n}$ is indeed primary, we proceed by induction on $n$, the base case $\a\derp{1}=\a$ being trivial. Let $n \geqslant 2$, and assume that $xy \in \a\derp{n}$, but $x \notin Q$. Since $\a\derp{n} \subseteq \a\derp{n-1}$, and $\a\derp{n-1}$ is $Q$-primary by induction, we necessarily have that $y \in \a\derp{n-1}$. By assumption, we have that $\delta(xy) \in\a\derp{n-1}$, therefore
\[
\ds (x^p+p\delta(x))\delta(y) = \delta(xy) - y^p\delta(x) \in \a\derp{n-1}.
\]
Because $p \in Q$, we have that $x^p+p\delta(x) \notin Q$, otherwise $x \in Q$. Since $\a\derp{n-1}$ is $Q$-primary, we conclude that $\delta(y) \in \a\derp{n-1}$, that is, $y \in \a\derp{n}$. 

(\ref{prop_pdiff_power_4}) follows from (\ref{prop_pdiff_power_1}) and (\ref{prop_pdiff_power_3}), since $\a^{(n)}$ is the smallest $Q$-primary ideal that contains $\a^n$.
\end{proof}

\begin{remark} Let $\a$ be a $Q$-primary ideal, with $p \in Q$. Note that $\a^n \subseteq \a\derp{n}$ is not true, in general, if we do not assume that $\a$ itself contains $p$.
\end{remark}
\begin{example}
Consider the ideal $\a=(4,x+2,y+2)$ inside the polynomial ring $R=\ZZ[x,y]$, and let $p=2$. Then $\a$ is $Q$-primary, with $Q=(2,x,y)$. However, we claim that there is a $p$-derivation $\delta$ for which  $\a^2 \not\subseteq \a\derp{2}$. Set $f=x+2$ and $g=y+2$, and let $\phi$ be the ring homomorphism 
\[
\xymatrixrowsep{1mm}
\xymatrixcolsep{5mm}
\xymatrix{
\phi: & \ZZ[x,y] \ar[rr] && \ZZ[x,y] \\
& h(x,y) \ar@{|->}[rr]&& h(x^2,y^2).
}
\]
Consider the associated $p$-derivation $\delta \in \Der_p(R)$, defined as $\delta(h) = \frac{\phi(h)-h^p}{p}$. Note that $\delta(2)=-1$, while $\delta(x)=\delta(y)=0$. Then we have
\[
\ds \delta(fg) = f^2\delta(g) + g^2 \delta(f) + 2\delta(f)\delta(g).
\]
Since $f^2,g^2 \in \a$, we have that $\delta(fg) \in \a$ if and only if $\delta(f)\delta(g) \in \a :_R (2) = Q$. We have
\[
\ds \delta(f) = \delta(x+2) = \delta(x)+\delta(2) + \frac{x^2+2^2 - (x+2)^2}{2} = -1 + \frac{-4x}{2} = -1-2x
\]
and similarly $\delta(g) = \delta(y+2) = -1-2y$. If follows that $\delta(f)\delta(g) \notin Q$. This shows that $\delta(fg) \notin \a$, and hence $fg \in \a^2 \smallsetminus \a\derp{2}$.
\end{example}

\begin{lemma} \label{pdiff-powers-localize} Let $p\in \ZZ$ be a prime, $S$ be a ring with a $p$-derivation $\delta$, and $Q \in \Spec(S)$ be a prime ideal containing $p$. Let $\delta_Q$ be the extension of $\delta$ to $R_Q$. Then 
$$Q\derp{n} S_Q = (Q S_Q)^{\lr{n}_{p}},$$
where the left-hand side is the $p$-differential power with respect to $\delta$ and the right-hand side is the $p$-differential power with respect to $\delta_Q$.
\end{lemma}

\begin{proof} It suffices to verify that $\frac{r}{1} \in Q\derp{n} R_Q$ if and only if $\frac{r}{1} \in (Q R_Q)^{\lr{n}_{p}}$, which is immediate from the fact that $(\delta_Q)^a(\frac{r}{1})=\frac{\delta^a(r)}{1}$ for all $a\geq 0$.
\end{proof}

We now make a key definition of this article. We combine the action of differential operators and $p$-derivations in order to control symbolic powers of prime ideals that contain a given prime integer. The definition we write is very general, but we will later focus on a more restrictive setting.

\begin{definition} \label{defn_difM} Let $\delta$ be a $p$-derivation on $S$. Let $I$ be an ideal of $S$. The \emph{$n$-th mixed differential power} of $I$ \emph{with respect to $\delta$} is 
\begin{align*}
I\difM{n} & :=\bigcap_{a+b\leq n+1} (I\derp{a})\dif{b}{B}\\
&=\{ f \in S \ | \ (\delta^s \circ \partial)(f) \in I  \ \text{for all} \ \partial \in D^t_{S|B} \ \text{with} \ s+t\leq n-1\}.
\end{align*}
\end{definition}

	Note that given $a+b\leq n+1$, computing $I\derp{a}$ involves applying $\delta^s$ with $s \leqslant a-1$, while taking $J\dif{b}{B}$ requires taking differential operators $\partial \in D^{b-1}_{S|B}$, so that overall, to compute $I\difM{n}$ we need to use combinations of differentials and $p$-derivations of order up to $(a-1) + (b-1) \leqslant n-1$.
	
	A word of caution about the order of operations: an element $f \in (I\derp{a})\dif{b}{B}$ is one such that $\partial(f) \in I\derp{a}$ for all $\partial \in D^{b-1}_{S|B}$, so that $\delta^s \circ \partial(f) \in I$ for all $s \leqslant a-1$.

\begin{remark}
	Note that the definition of $I\difM{n}$ depends, in principle, on both $\delta$ and $B$. However, we will show in Corollary~\ref{mixed powers independent of delta} that this definition is actually independent of the choice of $\delta$ for prime ideals in our main setting.
\end{remark}

\begin{proposition}\label{prop_mix_diff_power} If $I$ is an ideal, then $I\difM{n}$ is an ideal. Moreover, if $Q \in \Spec(S)$ is a prime ideal that contains $p$, we have the following properties:
\begin{enumerate}
\item \label{prop_mdiff_power_2a} $Q\difM{n} \subseteq Q\difM{m}$ if $n \geqslant m$.
\item \label{prop_mdiff_power_1} $Q^n \subseteq Q\difM{n}$ for all $n$.

\item \label{prop_mdiff_power_3} $Q\difM{n}$ is a $Q$-primary ideal.
\item \label{prop_mdiff_power_4} $Q^{(n)} \subseteq Q\difM{n}$.
\end{enumerate}
\end{proposition}

\begin{proof}
First of all, we show that $I\difM{n}$ is an ideal. Let $a,b$ be integers such that $a+b \leq n+1$. Then $I\derp{a}$ is an ideal by Proposition \ref{prop_pdiff_power}. As a consequence, we have that $\left(I\derp{a}\right)\dif{b}{B}$ is an ideal by Proposition \ref{prop_ordinary_dif} (\ref{prop_ordinary_dif_1}). Since $I\difM{n}$ is then defined as an intersection of ideals, the claim follows.

Part~(\ref{prop_mdiff_power_2a}) is immediate from the definition.

For part~(\ref{prop_mdiff_power_1}), we have that $\partial(Q^n) \subseteq Q^{n-t}$ for all $t \leq n-1$, and $\partial \in D^t_{S|B}$, by Proposition \ref{prop_ordinary_dif} (\ref{prop_ordinary_dif_2}). Since $Q^{n-t} \subseteq Q\derp{n-t}$ by Proposition \ref{prop_pdiff_power} (\ref{prop_pdiff_power_1}), it follows that $Q^n \subseteq \left(Q\derp{n-t}\right)\dif{t+1}{B}$. Because this holds for all $t \leq n-1$, we conclude that $Q^n \subseteq Q\difM{n}$.

Finally, parts (\ref{prop_mdiff_power_3}) and (\ref{prop_mdiff_power_4}) follow from the corresponding properties for differential powers and $p$-derivation powers. In fact, if $a$ and $b \in \NN$ are integers such that $a+b \leq n+1$, then observe that $Q\derp{a}$ is a $Q$-primary ideal by Proposition \ref{prop_pdiff_power} (\ref{prop_pdiff_power_3}). Therefore, the ideal $\left(Q\derp{a}\right)\dif{b}{B}$ is $Q$-primary by Proposition \ref{prop_ordinary_dif} (\ref{prop_ordinary_dif_3}). It then follows that $Q\difM{n} = \bigcap_{a+b\leq n+1} \left(Q\derp{a}\right)\dif{b}{B}$ is $Q$-primary, since it is a finite intersection of $Q$-primary ideals. Because the $n$-th symbolic power $Q^{(n)}$ is the smallest $Q$-primary ideal that contains $Q^n$, this concludes the proof.
\end{proof}

\begin{lemma}\label{mdiff-powers-localize} Let $\delta$ be a $p$-derivation on $R$. Let $Q \in \Spec(R)$ be a prime containing $p$. Then 
\[
\ds  Q\difM{n} R_Q = (Q R_Q)\difM{n},
\]
where $(Q R_Q)\difM{n}$ is the mixed differential power with respect to $\delta_Q$.
\end{lemma}
\begin{proof}
This is a direct consequence of Lemmas \ref{diff-powers-localize} and \ref{pdiff-powers-localize}.
\end{proof}

\begin{setting}\label{setting-containsP}
Let $p$ be a prime. Let $A=\ZZ$ or a DVR with uniformizer $p$. Let $R$ be an essentially smooth $A$-algebra that has a $p$-derivation $\delta$. Let $Q$ be a prime ideal of $R$ that contains $p$, and assume that the field extension $A/pA \hookrightarrow R_Q/QR_Q$ is separable.
\end{setting}

We note that polynomial rings over $\ZZ$, over Witt vectors of perfect fields, or over complete unramified DVRs of mixed characteristic and with perfect residue field satisfy Setting~\ref{setting-containsP}, by Proposition~\ref{existence-p-der}.

\begin{proposition} \label{prop_containment_m} In the context of Setting~\ref{setting-containsP}, suppose further that $(R,\m)$ is local. Then $\m\difM{n}=\m^n$.
\end{proposition} 

\begin{proof}
If $p,y_1,\ldots,y_d$ is a minimal generating set for the maximal ideal $\m$, then, by Lemma~\ref{free-basis}, $dy_1, \ldots, dy_d$ is part of a free basis for $\Omega_{R|A}$. It follows from Theorem~\ref{EGA_existence_diff_operators} that, for every $\alpha = (\alpha_1,\ldots,\alpha_d) \in \NN^d$, there exists a differential operator $D_\alpha \in D^{|\alpha|}_{R|A}$ such that $D_\alpha(y^\beta) = {\beta \choose \alpha} y^{\beta-\alpha}$ for all $\beta = (\beta_1,\ldots,\beta_d) \in \NN^d$ such that $\beta_i\geq \alpha_i$ for all $i$. Furthermore, we have $D_\alpha(y^\beta) = 0$ for all $\beta$ for which $\beta_i<\alpha_i$ for some $i$.
	
The containment $\m^n \subseteq \m\difM{n}$ always holds. For the converse, assume that there exists $f \in \m\difM{n}$ of order $s < n$, meaning that $f \in \m^s$, $f \notin \m^{s+1}$. Then, we can write $f = \sum_i u_i p^{s-|\alpha_i|} y^{\alpha_i} + g$ for some units $u_i \in R$, some $\alpha_i$ with $|\alpha_i| \leq s$, and some $g \in \m^{s+1}$. Fix some multi-index $\alpha \in \NN^d$ and a unit $u$ such that $u p^{s-|\alpha|}y^\alpha$ appears in the expression of $f$ as above, with $|\alpha|$ maximal as such. Observe that multiplying by a unit does not affect whether or not $f$ belongs to the ideals $\m\difM{n}$ and $\m^n$. Therefore, after multiplying by $u^{-1}$, we may assume that $p^{s-|\alpha|}y^\alpha$ appears in the support of $f$, with $|\alpha|$ maximal. Consider the corresponding differential operator $D_\alpha \in D^{s}_{R|A}$. Let $t=s-|\alpha|$ and observe that $D_{\alpha} \left( p^{t}y^\alpha \right) = p^tD_\alpha(y^\alpha) = p^t$. Recall that $p^t\notin\m\derp{t+1}$ by Remark~\ref{p-der-p-powers}. For the remaining $\alpha_i \ne \alpha$ with $|\alpha_i| = |\alpha|$, we have 
\[D_{\alpha} \left( u_i p^{t} y^{\alpha_i} \right) = p^t D_{\alpha} \left( u_i y^{\alpha_i} \right) \in p^t\left( y_1, \ldots, y_d \right) \subseteq \m^{t+1} \subseteq \m\derp{t+1},\] by Corollary~\ref{fixing the mix up} and Proposition~\ref{prop_pdiff_power}~(\ref{prop_pdiff_power_1}). For $\alpha_i$ such that $|\alpha_i| < |\alpha|$, we must have $s-|\alpha_i| \geq t+1$, so that \[D_\alpha(u_ip^{s-|\alpha_i|}y^{\alpha_i}) =p^{s-|\alpha_i|}D_\alpha(u_iy^{\alpha_i}) \in (p^{t+1}) \subseteq \m\derp{t+1},\]
again using Proposition \ref{prop_ordinary_dif}~(\ref{prop_pdiff_power_1}). Finally, note that $D_\alpha(g) \in \m^{t+1}$ by Proposition~\ref{prop_ordinary_dif}~(\ref{prop_ordinary_dif_2}), and thus $D_\alpha(g) \in \m\derp{t+1}$ as well. Combining these facts together, and using that $\m \derp{t+1}$ is an ideal, we obtain
\[
\ds D_{\alpha} \left(f \right) = p^t +  D_\alpha\left(\sum_{{\tiny \begin{array}{c} \alpha_i \ne \alpha \\ |\alpha_i|=|\alpha| \end{array}}} u_ip^ty^{\alpha_i}\right) + D_\alpha\left(\sum_{{\tiny \begin{array}{c} |\alpha_i|< |\alpha| \end{array}}} u_ip^{s-|\alpha_i|}y^{\alpha_i}\right) +D_\alpha(g) \notin \m\derp{t+1}.
\]
Thus, $f \notin (\m\derp{t+1})\dif{|\alpha|+1}{A}$, so that $f\notin \m\difM{s+1}$. This contradicts the assumption that $f \in \m\difM{n} \subseteq \m\difM{s+1}$, where the last containment follows from Proposition \ref{prop_pdiff_power} (\ref{prop_pdiff_power_2b}).
\end{proof}

\begin{theorem} \label{thm_mdiff=symb} In the context of Setting~\ref{setting-containsP}, one has the equality $Q^{(n)} = Q\difM{n}$.
\end{theorem}
\begin{proof}
Since both $Q^{(n)}$ and $Q\difM{n}$ are $Q$-primary, it is enough to show equality locally at $Q$. After localizing, we have that $Q^{(n)}R_Q \cong Q^nR_Q$. Moreover, Lemma~\ref{mdiff-powers-localize} implies that \[Q\difM{n}R_Q \cong (QR_Q)\difM{n}  = (QR_Q)^n\] by Proposition~\ref{prop_containment_m} applied to the local ring $(R_Q,QR_Q)$.
\end{proof}

\begin{corollary} 
In the context of Setting~\ref{setting-containsP}, suppose further that $R$ is smooth over $A$. Then
\[
\ds Q^{(n)}=\bigcap_{\m \in \mathcal{B}} \m^n,
\]
for all $n \geq 1$, where $\mathcal{B}= \Max\Spec(R) \cap \V(Q)$.
\end{corollary}
\begin{proof}
By Theorem~\ref{thm_mdiff=symb}, it is sufficient to show that $Q\difM{n} = \bigcap_{\m \in \mathcal{B}} \m^n$. Since $R$ is smooth over $A$ by assumption, we have that $R/pR$ is an algebra of finite type over the field $A/pA$. In particular, $R/pR$ is a Hilbert-Jacobson ring, and thus $Q = \bigcap_{\m \in \mathcal{B}} \m$. Observe that $Q\difM{n} \subseteq \m\difM{n}$ for all $\m \in \mathcal{B}$ is clear, since $Q \subseteq \m$. To prove the converse, let $f \in \bigcap_{\m \in \mathcal{B}} \m^n$, so that $f \in \m\difM{n}$ for all $\m \in \mathcal{B}$, by Proposition \ref{prop_containment_m}. Therefore $(\delta^a \circ \partial)(f) \in \bigcap_{\m \in \mathcal{B}} \m = Q$ for any given $p$-derivation $\delta$ and any differential operator $\partial \in D^{n-1-a}_{R|A}$. It follows that $f \in Q\difM{n}$, as desired. 
\end{proof}

Recall that the definition of mixed differential powers depends, a priori, on the chosen $p$-derivation. However, Theorem~\ref{thm_mdiff=symb} immediately gives that this is not the case in the context of Setting \ref{setting-containsP}. We record this fact in the following corollary. 
\begin{corollary}\label{mixed powers independent of delta} In the context of Setting~\ref{setting-containsP}, the ideal $Q\difM{n}$ is independent of the choice of $p$-derivation $\delta$. Moreover, 
\[
	Q^{(n)}=Q\difM{n} =\{ f \in R \ | \ (\delta \circ \partial)(f) \in I  \ \text{for all} \ \delta\in \Der^{a}_p(R), \ \partial \in D^b_{R|A} \ \text{with} \ a+b\leq n-1\}.
\]
\end{corollary}
\begin{proof}
Write $Q^{\lr{n}_\star}$ for the set on the right-hand side, and recall the following definition 
\[
\ds Q^{\lr{n}_{\Der_p}}=\{f\in R \ | \ \Der_p^a(f) \subseteq Q \ \text{for all} \ a\leq n-1\}.
\] 
Observe that $Q^{\lr{n}_\star} =\bigcap_{a+b\leq n+1} (Q^{\lr{a}_{\Der_p}})\dif{b}{A}$; in particular, $Q^{\lr{n}_\star}$ is an ideal by Proposition~\ref{prop_ordinary_dif}~(\ref{prop_ordinary_dif_1}) and Proposition~\ref{prop_pdiff_power}. By Proposition~\ref{prop_pdiff_power} (\ref{prop_pdiff_power_4}), we have that $Q^{(a)} \subseteq Q^{\lr{a}_{\Der_p}}$ for all $a$. Moreover, $Q^{\lr{a}_{\Der_p}}$ is $Q$-primary by Proposition~\ref{prop_pdiff_power}~(\ref{prop_pdiff_power_3}), and it follows from Proposition~\ref{prop_ordinary_dif}~(\ref{prop_ordinary_dif_4}) that $Q^{(n)} \subseteq \big(Q^{\lr{a}_{\Der_p}}\big)\dif{b}{A}$ for all $a+b \leq n+1$. Therefore, we have $Q^{(n)}\subseteq Q^{\lr{n}_\star}$. It is clear from the definitions that $Q^{\lr{n}_\star}\subseteq Q\difM{n}$ and, by Theorem \ref{thm_mdiff=symb}, we finally conclude that $Q^{\lr{n}_\star}=Q\difM{n}=Q^{(n)}$. 
	\end{proof}
	
We point out that, in the Definition \ref{defn_difM} of mixed differential powers, the order in which $p$-derivations and differential operators are performed is crucial, as the following example illustrates.
	
\begin{example}
Let $R=\ZZ_p[x]$, and $Q=(p,x)$. Let $\phi:R \to R$ be the lift of Frobenius that satisfies $\phi(x)=x^p$. This induces a $p$-derivation $\delta$ on $R$ such that $\delta(x)=0$. Also, note that $D^1_{R|\ZZ_p} = R \oplus R \frac{d}{dx}$ by Lemma~\ref{EGA_existence_diff_operators}. One can check by direct computation that $\frac{d^2}{dx^2}(px)$, $(\frac{d}{dx}\circ \delta)(px)$, and $\delta^2(px)$ all belong to $Q$. However, $px \notin Q^{(3)} = Q^3$, where the last equality holds because $Q$ is a maximal ideal. Note that, consistently with Theorem \ref{thm_mdiff=symb}, we have that $px \notin Q\difM{3}$, since $(\delta \circ \frac{d}{dx})(px) \notin Q$.
\end{example}

As in the equicharacteristic case, the essential smoothness hypothesis is important. Many nonsmooth algebras do not admit $p$-derivations, in which case the mixed differential powers are not defined. However, even in nonsmooth algebras with $p$-derivations, the hypothesis is necessary.

\begin{example}
	Let $R=\ZZ_p[x,y,z]/(y^2-xz)$. The lift of the Frobenius $\Phi(f(u,v))=f(u^p,v^p)$ on $\ZZ_p[u,v]$ restricts to a lift of the Frobenius on $\ZZ_p[u^2,uv,v^2]\cong R$, so $R$ admits a $p$-derivation. In particular, mixed differential powers are defined. Take $Q=(p,x,y)$ and $\m=(p,x,y,z)$. Then, $x\in Q^{(2)} \smallsetminus \m^2$. However, it is evident from Definition~\ref{defn_difM} that $Q\difM{2}\subseteq \m\difM{2}$. Thus, the conclusion of Theorem~\ref{thm_mdiff=symb} cannot hold.
\end{example}

The following example shows that the separability hypothesis in Theorem~\ref{thm_mdiff=symb} is necessary. Given that a similar hypothesis is required for the Zariski-Nagata theorem in equicharacteristic $p>0$, this is to be expected.

\begin{example}
	Let $A=\ZZ[t]_{(p)}$, $R=\ZZ[t,x]_{Q}$, where $Q = \left( p, x^p-t \right)$. We note that these satisfy all of the conditions of Setting~\ref{setting-containsP}, except the separability assumption on $A/pA \hookrightarrow R_Q / Q R_Q$. In this example, $Q^{(2)}=Q^2$. However, we will show that there exists a $p$-derivation on $R$ such that $Q\difM{2} \neq Q^{(2)}$. Set $w=x^p-t$. We can write $\ZZ[t,x] = \ZZ[w,x]$. Then, the map
	\[\psi(f(w,x)) = \frac{f(w^p,x^p)-f(w,x)^p}{p} \]
	is a $p$-derivation such that $\psi(w)=0$. Since $R$ is a localization of $\ZZ[t,x]$, we have by Proposition~\ref{existence-p-der}~(\ref{p-der-localize}) that there is a $p$-derivation $\delta$ on $R$ that extends $\psi$. In particular, $\delta(w)=0\in Q$. Also, $\frac{d}{dx}(w)=px^{p-1} \in Q$. Thus, $w\in Q\difM{2} \smallsetminus Q^{(2)}$.
\end{example}

\section*{Acknowledgments}
We would like to thank Alexandru Buium for a helpful correspondence, and Bhargav Bhatt for useful conversations and suggestions. We thank the anonymous referees for many helpful comments and suggestions. We thank Holger Brenner and Luis N\'u\~nez-Betancourt for allowing us to include a proof of Lemma~\ref{diff-powers-localize}. We thank Rankeya Datta for pointing out to us a mistake in a previous proof of Lemma \ref{free-basis}, due to a wrong claim about the local structure of essentially smooth algebras over a field or a DVR. The third author was partially supported by NSF Grant DMS~\#1606353.

\bibliographystyle{alpha}
\bibliography{References}

\end{document}